\documentclass[10pt,reqno,final]{amsart}
\usepackage{amsmath,amssymb,amsthm}
\usepackage{mathrsfs}
\usepackage{pifont}
\usepackage[english]{babel}
\usepackage{graphicx}
\usepackage{subfig}
\usepackage{hyperref}
\usepackage{enumitem}
\usepackage[notcite,notref]{showkeys}
\usepackage{color}
\usepackage{tabularx}
\usepackage{booktabs}
\usepackage{array}
\usepackage{makecell}
\usepackage{multirow,multicol}
\usepackage{caption}
\usepackage{algorithm,algorithmicx}
\usepackage{algpseudocode}
\usepackage{fancybox}

\usepackage[all]{xy}

\usepackage[numbers,sort&compress]{natbib}

\allowdisplaybreaks

\theoremstyle{plain}
\newtheorem{theorem}{Theorem}[section]
\newtheorem{lemma}{Lemma}[section]

\newtheorem{proposition}{Proposition}[section]

\theoremstyle{definition}
\newtheorem{definition}{Definition}[section]
\newtheorem{example}{Example}

\theoremstyle{remark}

\numberwithin{equation}{section}
\numberwithin{figure}{section}
\numberwithin{table}{section}

\graphicspath{{./figures/}}

\title[Persistence Minimal Free Lie Model]{Persistence Minimal Free Lie Model}

\author[Siheng Yi]{Siheng Yi}

\thanks{${}^{*}$Siheng Yi (E-mail: \texttt{12131237@mail.sustech.edu.cn})}
\thanks{${}^{}$Department of Mathematics, Southern University of Science and Technology}

\date{\today}

\subjclass[2020]{55N31, 55P62}

\keywords{persistence modules, interleaving distance, persistence minimal free Lie models}

\begin{document}

\begin{abstract}
  The minimal Quillen model is a free Lie model for rational spaces proposed by Quillen. 
  Meanwhile, persistence modules are theoretical abstractions of persistent homology. 
  In this paper, we integrate the ideas of rational homotopy theory and persistence modules to construct the persistence minimal Quillen model and discuss its stability. 
  Our results provide a new algebraic framework for topological data analysis, which is more refined compared to directly computing the homology groups of the filtration of simplicial complexes. 
  Furthermore, the stability results for persistence minimal Lie models ensure that our model is well-founded. 
\end{abstract}

\maketitle

\section{Introduction}
\subsection{background}
Persistence modules\cite{zomorodian2004computing, chazal2016structure, oudot2015persistence}, as algebraic structures encoding the evolution of topological features across scales, have become central to the mathematical framework of topological data analysis(TDA)\cite{edelsbrunner2002topological, zomorodian2004computing, carlsson2009topology, cohen2005stability}. 
Their development, applications, and theoretical richness bridge pure mathematics, computational topology, and data science. 

Rational homotopy theory is a branch of algebraic topology that studies the homotopy-theoretic properties of spaces by rationalizing their homotopy groups and focusing on computational methods. 
It simplifies ordinary homotopy theory by discarding torsion information and retaining only rational coefficients, enabling explicit calculations for spaces like simply connected finite CW complexes. 
A key innovation in this field is Daniel Quillen's introduction of free Lie models, constructed through differential graded free Lie algebras, which encode the rational homotopy type of a space via its differential forms and Lie bracket. 

In persistent homology, if we have already determined the method for constructing simplicial complexes from point clouds, then the remaining issue is to establish algebraic models for these simplicial complexes. 
Currently, the most frequently used algebraic model is the homology groups over some coefficient field $\Bbbk$ for simplicial complexes. 
When we specify the coefficient field to be a field of characteristic 0, we can employ rational homotopy theory to establish a more refined algebraic model for simplicial complexes. 

In this paper, we first introduce our main results in Section 1. 
In Section 2, we will review some basic knowledge about the rational homotopy theory and persistence modules. 
Finally, we will fully elaborate on our results and proof in Section 3. 

\subsection{Main Results}
Before presenting our results, we need to establish some notation. 
\begin{itemize}
    \item $\textbf{DGL}$ is the category of connected differential graded Lie algebras over $\mathbb{Q}$,
    \item \textbf{Ho}($\textbf{DGL}$) is the homotopy category of $\textbf{DGL}$,
    \item $\textbf{Vec}$ is the category of finite-dimensional vector spaces over $\mathbb{Q}$,
    \item $\textbf{grVec}$ is the category of graded finite-dimensional vector spaces over $\mathbb{Q}$,
    \item $\textbf{Top}$ is the category of simply-connected rational spaces of finite type,
    \item $d_{HI}$ is the homotopy interleaving distance\cite{blumberg2023universality}. 
\end{itemize}

In this paper, we define persistence minimal Quillen models, which are persistence minimal free Lie models for rational $\mathbb{R}$-spaces, where $\textbf{Top} $ is the category of simply connected rational spaces of finite type and a rational $\mathbb{R}$-space is a functor $(\mathbb{R},\leq) \to \textbf{Top} $. 
Meanwhile, we prove the existence of persistence minimal Quillen models, that is
\vspace{0.3cm}
\begin{theorem}
    For any rational $\mathbb{R}$-space $\mathbb{X}: (\mathbb{R},\leq) \to \textbf{Top} $, there exists a persistence minimal Quillen model $M_{Qui}(\mathbb{X}):(\mathbb{R},\leq) \to \textbf{Ho}(\textbf{DGL})$ such that $M_{Qui}(\mathbb{X})_t$ is a minimal Quillen model of $\mathbb{X}_r$ and $M_{Qui}(\mathbb{X})(s\leq t)$ is a Lie representative of $\mathbb{X}(s\leq t)$ up to weak equivalences. 
\end{theorem}
\vspace{0.3cm}

Persistence minimal Quillen models 
As persistence minimal Quillen models, which is a type of persistence modules, we can consider its stability results, and thus we have proven the following theorem.
\vspace{0.3cm}
\begin{theorem}\label{12}
    For any rational $\mathbb{R}$-spaces $\mathbb{X}$ and $\mathbb{Y}$, we have 
    \begin{itemize}
        \item $d_{I}^{\textbf{Ho}(\textbf{DGL})}(M_{Qui}(\mathbb{X}),M_{Qui}(\mathbb{Y})) \leq d_{HI}(\mathbb{X},\mathbb{Y}) \leq d_{I}(\mathbb{X}, \mathbb{Y})$
        \item $\begin{array}{ll}
        d_I^{\textbf{grVec}}(\pi_*(\mathbb{X}), \pi_*(\mathbb{Y}) )&=d_I^{\textbf{grVec}}(H_* \circ M_{Qui}(\mathbb{X}),H_* \circ M_{Qui}(\mathbb{Y})) \\
        &\leq d_{I}^{\textbf{Ho}(\textbf{DGL})}(M_{Qui}(\mathbb{X}),M_{Qui}(\mathbb{Y}))
        \end{array}$
        \item $d_I^{\textbf{grVec}}(H_*(\mathbb{X}), H_*(\mathbb{Y}) )=d_I^{\textbf{grVec}}(\mathbb{V},\mathbb{W}) \leq d_{I}^{\textbf{Ho}(\textbf{DGL})}(M_{Qui}(\mathbb{X}),M_{Qui}(\mathbb{Y}))$
    \end{itemize}
\end{theorem}
\vspace{0.3cm}

\section{Preliminaries}

\subsection{Interleaving Distance}
Persistence modules are the categorization\cite{bubenik2014categorification} of persistent homology. 
In general, a persistence module can be defined as a functor $F: \mathcal{P} \to C$, in which $\mathcal{P}$ is a poset, where the $\text{ob }\mathcal{P}$ is the set $\mathcal{P}$, and morphisms are the partial order of $\mathcal{P}$, and $C$ is any abelian category. 
It is obvious that the category $\mathcal{P}$ is thin. 
A category is said to be thin if, for every pair of objects in this category, there exists at most one morphism from $a$ to $b$. 
In specific studies, the poset $\mathcal{P}$ that we usually consider is $(\mathbb{R},\leq)$ or $(\mathbb{Z},\leq)$. 
In category $\mathbb{R}$, the object is a real number $r\in \mathbb{R}$, and the morphism $r\to s$ exists if and only if $r\leq s$. 
Similarly, we can define the category $(\mathbb{Z},\leq)$.

In this paper, persistence modules we discuss are functors $(\mathbb{R},\leq) \to C$, where $C$ can be the category $\textbf{Vec}$, the category $\textbf{DGL}$, or the category $\textbf{Top} $. 

In this subsection, we will introduce morphisms and 'distance' between persistence modules, that is, the interleaving distance $d_I$. 
The interleaving distance\cite{escolar2023interleavings} between persistence modules can be seen as a generalization of the bottleneck distance $d_B$ between persistence diagrams\cite{edelsbrunner2002topological}. 
Note that unless otherwise specified, the persistence modules considered are always functors $\mathbb{R} \to C$, in which $C$ is any abelian category. 
For a persistence module $\mathbb{X}:\mathcal{P}\to C$, we denote $\mathbb{X}(a\leq b)$ as $\mathbb{X}_{a\leq b}$ and denote $\mathbb{X}(a)$ as $\mathbb{X}_a$. 

\vspace{0.3cm}
\begin{definition}
    For persistence modules $\mathbb{X}$ and $\mathbb{Y}$, a morphism between $\mathbb{X}$ and $\mathbb{Y}$ is a natural transformation between $\mathbb{X}$ and $\mathbb{Y}$, $f: \mathbb{X}\Rightarrow \mathbb{Y}$ that is denoted as $f:\mathbb{X} \to \mathbb{Y}$.
\end{definition}
\vspace{0.3cm}
The collection of all functors from $\mathcal{P}$ to $C$ and all natural transformations between the functors is the category $C^{\mathcal{P}}$.

If $\mathcal{P}=\mathbb{R}$, we may think that persistence modules depict the evolution of objects in $C$ over time. 
For instance, if persistence modules $\mathbb{X}, \mathbb{Y}$ satisfying $\mathbb{X}(t)=\mathbb{Y}(t+\delta)$ for some constant $\delta$, then $\mathbb{X}$ and $\mathbb{Y}$ are same by shifting time $\delta$. 
However, there is no isomorphism between persistence modules $\mathbb{X}$ and $\mathbb{Y}$, even morphism generally. 
Then, we need to expand the notations of morphisms and isomorphisms between persistence modules to the new version that contains the information of $\epsilon$-shifting.

For $\delta \geq 0$, we define that the $\delta$-interleaving category $I^{\delta}$ is the thin category such that $\text{ob }I^{\delta} := \mathbb{R} \times \{0,1\}$ and there is the morphism $(r,i)\to (s,j)$ if and only if either

(1) $r+\delta \leq s$, or

(2) $i=j$ and $r\leq s$.\\
There exist two functors 
$$E^0,E^1:\mathbb{R} \to I^{\delta}$$
mapping $r \in \mathbb{R}$ to $(r,0)$ and $(r,1)$, respectively.
\vspace{0.3cm}
\begin{definition}
    Let $C$ be any category and $\mathbb{X},\mathbb{Y}: \mathbb{R} \to C$ be any two functors.
    A $\delta$-interleaving between $\mathbb{X}$ and $\mathbb{Y}$ is a functor
    $$Z:I^{\delta} \to C$$
    satisfying $Z \circ E^0=\mathbb{X}$ and $Z\circ E^1=\mathbb{Y}$. 
\end{definition}
\vspace{0.3cm}
We call persistence modules $\mathbb{X},\mathbb{Y}:\mathbb{R}\to C$ are $\delta$-interleaved, if there exists a functor $Z:I^{\delta} \to C$ that is a $\delta$-interleaving between $\mathbb{X}$ and $\mathbb{Y}$.

Let $\mathbb{X}(\delta):\mathbb{R} \to C$ be the functor by shifting $\mathbb{X}$ downward by $\delta$, i.e., $\mathbb{X}(\delta)_r:=\mathbb{X}_{r+\delta}$ and $\mathbb{X}(\delta)_{r\leq s}:=\mathbb{X}_{r+\delta \leq s+\delta}$ for all $r\leq s \in \mathbb{R}$. 
Similarly, $f(\delta):\mathbb{X}(\delta) \to \mathbb{Y}(\delta)$ is defined by $f(\delta)_{t} := f_{t+\delta}$, where $f:\mathbb{X}\to \mathbb{Y}$ is a morphism between persistence modules. 
Specially, we define the morphism $\phi^{\mathbb{X},\delta}: \mathbb{X} \to \mathbb{X}(\delta)$ for any $\mathbb{X}:\mathbb{R} \to C$, in which $\phi^{\mathbb{X},\delta}_t=\mathbb{X}_{t\leq t+\delta}$. 
A $\delta$-interleaving $Z$ between $\mathbb{X}$ and $\mathbb{Y}$ is characterized by a pair of natural transformations $f:\mathbb{X}\to \mathbb{Y}(\delta)$ and $g:\mathbb{Y} \to \mathbb{X}(\delta)$ , satisfying the compatibility conditions $g(\delta) f=\phi^{\mathbb{X},2\delta}$ and $f(\delta) g=\phi^{\mathbb{Y},2\delta}$. 
On the other hand, $Z: I^{\delta} \to C$ is entirely determined by these natural transformations, which are referred to as $\delta$-interleaving morphisms. 
When $\delta=0$, these morphisms reduce to a pair of mutually inverse natural isomorphisms. 

\vspace{0.3cm}
\begin{definition}
    We define the interleaving distance $d_I$ as a binary function
    $$d_I:\text{ob }C^{\mathbb{R}}\times \text{ob }C^{\mathbb{R}} \to [0,\infty],$$
    by taking 
    $$d_I(\mathbb{X},\mathbb{Y}):=\text{inf } \{\delta\ | \mathbb{X} \text{ and } \mathbb{Y} \text{ are }\delta\text{-interleaved}\}.$$
\end{definition}
\vspace{0.3cm}
It is straightforward to verify that if $\mathbb{X}$ and $\mathbb{W}$ are $\delta$-interleaved, and $\mathbb{W}$ and $\mathbb{Y}$ are $\epsilon$-interleaved, then $\mathbb{X}$ and $\mathbb{Y}$ are $\delta + \epsilon$-interleaved. 
Thus, we know that $d_I$ satisfies the triangle inequality. 
Therefore, the $d_I$ is obviously a pseudo-distance. 
What's more, if $\mathbb{X},\mathbb{X}',\mathbb{Y} \in \text{ob } C^{\mathbb{R}}$ with $\mathbb{X} \cong \mathbb{X}'$, then $d_I(\mathbb{X},\mathbb{Y})=d_I(\mathbb{X}' ,\mathbb{Y})$, so function $d_I$ defines a pseudo-distance on the isomorphism classes of objects in the category $C^{\mathbb{R}}$.

One of the most useful aspects of the categorical view of interleavings is that if we apply a functor to $\delta$-interleaving, then the resulting diagrams are also $\delta$-interleaving. That is, 
\vspace{0.3cm}
\begin{proposition}{\cite{Bubenik_2014}}
    Let $\mathbb{X},\mathbb{Y}:\mathbb{R} \to C$ be two persistence modules and $H:C \to D$ be a functor. 
    If $\mathbb{X}$ and $\mathbb{Y}$ are $\delta$-interleaved, then so are $H\mathbb{X}$ and $H\mathbb{Y}$. Thus,
	$$d_I(H\mathbb{X},H\mathbb{Y}) \leq d_I(\mathbb{X},\mathbb{Y}).$$
\end{proposition}
\vspace{0.3cm}
The process of composition of functors can be seen as the process of processing information, and information may be lost after processing, so the difference between two persistence modules may be reduced. 
From this perspective, it is also easy to understand the actual meaning of the previous proposition. 
Meanwhile, there are scholars studying similar topics in this discussion, which is the change of interleaving distance when persistence modules composite some functors\cite{ginot2019multiplicative,lupo2022persistence}.

\subsection{Quillen Models}
A simply connected space $X$ is called a rational space if $X$ satisfies one of following equivalent conditions(Theorem 9.3 of \cite{felix2012rational}):
    \begin{itemize}
        \item $\pi_{*}(X) \cong \pi_{*}(X) \otimes_{\mathbb{Z}} \mathbb{Q}$
        \item $H_{*}(X,pt;\mathbb{Z}) \cong H_{*}(X,pt;\mathbb{Z})\otimes_{\mathbb{Z}} \mathbb{Q}$
        \item $H_{*}(\Omega X,pt;\mathbb{Z}) \cong H_{*}(\Omega X,pt;\mathbb{Z})\otimes_{\mathbb{Z}} \mathbb{Q}$
    \end{itemize}
If $H_i(X,pt;\mathbb{Z})\otimes_{\mathbb{Z}} \mathbb{Q}$ is a finitely dimensional vector space for all $i\in \mathbb{N}$, we call $X$ is of finite type.
\vspace{0.3cm}
\begin{definition}
    For a simply connected space $X$, a rationalization of $X$ is a continuous map $\varphi: X \to X_{\mathbb{Q}}$ satisfying that $\varphi$ induces an isomorphism
    $$\pi_{*}(X) \otimes_{\mathbb{Z}} \mathbb{Q} \to \pi_{*}(X_{\mathbb{Q}}),$$
    where $X_{\mathbb{Q}}$ is a simply connected rational space. 
\end{definition}
\vspace{0.3cm}
For any simply connected topological space $X$, we can always find a rational space $X_{\mathbb{Q}}$ such that $X_{\mathbb{Q}}$ is the rationalization of $X$. 
\vspace{0.3cm}
\begin{theorem}{\cite{felix2012rational}}
    (i) Let $X$ be a simple connected space. 
    Then there exists a relative CW complex $(X_{\mathbb{Q}} , X)$ that lacks 0-dimensional and 1-dimensional cells so that the inclusion $\varphi:X\to X_{\mathbb{Q}} $ is a rationalization. 
        
    (ii) Given $(X_{\mathbb{Q}} , X)$ as described in (i) and $Y$ as any simply connected rational space.
    For any continuous map $f: X \to Y$, we may extend $f$ to a continuous map $g:X_{\mathbb{Q}} \to Y$. 
    Furthermore, if $g':X_{\mathbb{Q}}  \to Y$ extends $f':X\to Y$ then any homotopy between $f$ and $f'$ can be extended to a homotopy between $g$ and $g'$. 
        
    (iii) The rationalization specified in (i) is unique up to homotopy equivalence relative to $X$. 
\end{theorem}
\vspace{0.3cm}
The theorem told us that every simply connected space can be rationalized and every continuous map $\varphi:X \to Y$ between simply connected spaces can induce the continuous map $\tilde{\varphi}:X_{\mathbb{Q}} \to Y_{\mathbb{Q}}$. 
    
In this subsection, we will focus on the category $\textbf{Top} $ of simply connected rational spaces of finite type, and objects in $\textbf{Top} $, that is simply connected rational spaces of finite type. Therefore, unless otherwise stated, all topological spaces encountered in this paper are assumed to be simply connected rational spaces of finite type, and all numerical fields involved are assumed to be the field of rational numbers, $\mathbb{Q}$.
    
Specifically, we may notice that for any $X \in \text{ob }\textbf{Top} $, $\pi_{*}(X)$ is a vector space over $\mathbb{Q}$. 
Then for a functor $\mathbb{X}: (\mathbb{R},\leq) \to \textbf{Top} $, $\pi_{*}(\mathbb{X}), H_{*}(\mathbb{X}), H^{*}(\mathbb{X}) : (\mathbb{R},\leq) \to \textbf{grVec} $ are persistence modules, which is the most commonly encountered persistence module. 
In rational homotopy theory, we have more refined algebraic models than homotopy groups and homology groups, minimal Sullivan models, and minimal free Lie models.

In Quillen's paper\cite{quillen1969rational}, Quillen defined and used a sequence of functors that are Quillen equivalent, respectively, to assign to a simply connected rational space of finite type a differential graded Lie algebra ($dgl$), 
$$X \mapsto \lambda X. $$
We call the functor $$\lambda : \textbf{Top}  \to \textbf{DGL}$$ Quillen functor where $\textbf{DGL}$ is the category  of connected $dgl$, that is $L=\{L_i\}_{i>0}$.

Before starting a detailed introduction to the Quillen model of rational spaces, we will first introduce the functors defined by Quillen and their main properties in homotopy theory.

We need to recall some notions of coalgebras and Lie algebras. 
\vspace{0.3cm}
\begin{definition}
    A graded coalgebra $C$ consists of a graded module $C$ equipped with two degree-preserving linear maps, one of which is called the comultiplication $\Delta: C \to C \otimes C $, and the other is referred to as the augmentation $\epsilon: C \to \mathbb{Q}$. 
    These maps satisfy the coassociativity condition $(\Delta \otimes id)\Delta = (id \otimes \Delta)\Delta$ and the counit condition$(id\otimes \epsilon)\Delta=(\epsilon \otimes id)\Delta ={id}_C$. 
\end{definition}
\vspace{0.3cm}
A graded coalgebra is called cocommutative if 
$$\tau \Delta = \Delta$$
where $\tau : C\otimes C \to C \otimes C$ is the involution $a\otimes b \mapsto {(-1)}^{\text{deg }a \text{ deg }b} b\otimes a$.
We call a graded coalgebra co-augmented by the choice of an element $1 \in C_0$ so that $\epsilon (1)=1$ and $\Delta(1)=1\otimes 1$. 
We can also say that co-augmentation is an embedding $\mathbb{Q} \hookrightarrow C$.  
For such coalgebra $C$, we write $\bar{C}=\text{Ker } \epsilon$, so that $C=\mathbb{Q}\oplus \bar{C}$ and define $\bar{\Delta}: \bar{C} \to \bar{C} \otimes \bar{C}$ with $\bar{\Delta}c=\Delta c- c\otimes 1 - 1\otimes c$. 
\vspace{0.3cm}
\begin{example}
    The coalgebra $\Lambda V$ is an instructive example, where comultiplication $\Delta$ is explicitly defined by the formula $\Delta v = v\otimes 1 + 1\otimes v, \ v \in V$. 
    And the augmented by $\epsilon :\Lambda^+ V \to 0,\ 1\mapsto 1$ and co-augmented by $\mathbb{Q}=\Lambda^0 V$. 
\end{example}
\vspace{0.3cm}
\begin{definition}
    A graded Lie algebra $L$ consists of a graded vector space $L=\{L_i\}_{i \in \mathbb{Z}}$ and a linear map of degree zero, $L \otimes L \to L$, denoted by $x\otimes y \mapsto [x,y]$ which satisfies the following conditions:
    \begin{itemize}
        \item $[x,y]=-{(-1)}^{\text{deg }x \text{ deg }y}[y,x]$
        \item $[x,[y,z]]=[[x,y],z]+{(-1)}^{\text{deg }x \text{ deg }y}[y,[x,z]]$
    \end{itemize}
    The product $[\ ,\ ]$ is called the Lie bracket. 
\end{definition}
\vspace{0.3cm}
We say a linear map of degree $k$, $\theta:L\to L$, is a A derivation of $L$ of degree $k$ if $\theta [x,y]=[\theta x,y]+{(-1)}^{k\text{ deg }x}[x,\theta (y)]$. 
\vspace{0.3cm}
\begin{example}
    Let $V$ be a graded vector space. 
    The tensor algebra $TV$ on $V$ carries a natural graded Lie algebra structure via the bracket operation $[x,y]:=xy-{(-1)}^{\text{deg }x \text{ deg }y}yx$. 
    Then, the free graded Lie algebra $\mathbb{L}_V$ is defined as the smallest graded Lie subalgebra of $TV$ containing $V$. 
    This object satisfies a universal property: any degree-preserving linear map $f: V \to L$ into another graded Lie algebra $L$ may extend uniquely to a graded Lie algebra homomorphism $\mathbb{L}_{V} \to L$. 

    The free graded Lie algebra $\mathbb{L}_V$ naturally inherits a grading structure from the tensor algebra $TV$, which decomposes as the direct sum $\bigoplus_{k=0}^{\infty} T^{k}V$. Here, each homogeneous component $T^{k}V$ consists of tensors of degree $k$. 
    Since $\mathbb{L}_V$ is generated by iterated Lie brackets of elements in $V$, its elements can be stratified by bracket length, defined as the number of generators (from $V$) involved in their construction. 
    
    \begin{itemize}
        \item $\mathbb{L}_{V}=\bigoplus_{k \geq 1}(\mathbb{L}_{V} \cap T^k V);$
        \item $x\in\mathbb{L}_V$ has bracket length $k$ if and only if $x\in \mathbb{L}_V^k := \mathbb{L}_{V}\cap T^k V$.
    \end{itemize}
    Then we may decompose $\mathbb{L}_V=\bigoplus_{i \geq 1} \mathbb{L}_V^i$the differential $d=d_0+d_1+\cdots$, in which $d_k:V \to \mathbb{L}_V \cap T^{k+1} V$. 

    For any free Lie algebra $(\mathbb{L}_V,d=d_0+\cdots)$, if $d_0=0$, then we call it \textbf{minimal}. 
\end{example}
\vspace{0.3cm}
Next, we will review the two functors $C_* : \textbf{DGL} \to \textbf{CDGC} $ and $\mathcal{L} : \textbf{CDGC} \to \textbf{DGL}$ where \textbf{CDGA} is the category of 1-connected cocommutative differential graded coalgebras ($cdgc$), which played important roles in Quillen's work\cite{quillen1969rational}.

Suppose that $(L,d_L)$ is a differential graded Lie algebra. 
The coderivations in $\Lambda sL$, where $sL$ denotes the shift of degrees that is $(sL)_i=L_{i-1}$ for all $i$, are determined by the differential $d_L$ and the Lie bracket $[\ ,\ ]:L\otimes L \to L$ 
$$d_0(sx_1\wedge \cdots \wedge sx_k)=-\sum_{i=1}^{k} {(-1)}^{n_i} sx_1\wedge \cdots \wedge sd_Lx_i \wedge \cdots \wedge sx_k, $$
and 
$$d_1(sx_1\wedge \cdots \wedge sx_k)= \sum_{1\leq i < j \leq k}{(-1)}^{\text{deg }x_i +1}{(-1)}^{n_{ij}} s[x_i,x_j]\wedge sx_1 \cdots s\hat{x}_i \cdots s\hat{x}_j \cdots sx_k $$
where $n_i=\sum_{j<i} \text{deg } sx_j$, and $sx_1 \wedge \cdots \wedge sx_k= {(-1)}^{n_{ij}} sx_i \wedge sx_j \wedge sx_1 \cdots s\hat{x}_i \cdots s\hat{x}_j \cdots \wedge sx_k$. (Here, symbol $\hat{\ }$ means 'deleted'. ) 

By simple computation, we can know that $d=d_0+d_1$ is a coderivation. 
In other words, $(\Lambda sL,d=d_0+d_1)$ is a differential graded coalgebra. 
\vspace{0.3cm}
\begin{definition}
    The Cartan-Eilenberg-Chevalley construction on a $dgl$ $(L,d_L)$ is the $cdgc$ $C_*(L,d_L)=(\Lambda sL,d=d_0+d_1)$. 
\end{definition}
\vspace{0.3cm}
The functor $C_*$ assigns a $dgl$ $(L,d_L)$ a $cdgc$ $(\Lambda sL,d)$, and if $E=\{E_i \}_{i>0} $ and $L=\{L_i \}_{i>0} $, then $\varphi : E \to L$ is a quasi-isomorphism if and only if $C_*(\varphi)$ is a quasi-isomorphism\cite{felix2012rational}. 

There are some methods for constructing free Lie algebras, but we will introduce one that is closely related to $C_*$, Quillen's functor $\mathcal{L}$, which is the analog of the cobar construction. 

Let $(C,d)=(\bar{C},d)\oplus \mathbb{Q}$ be any co-augmented $cdgc$. 
By the cobar consturction, $\Omega C=Ts^{-1}\bar{C}$. 
The differential has the form $d=d_0+d_1$ with $d_0: s^{-1}\bar{C} \to s^{-1}\bar{C}$ and $d_1:s^{-1}\bar{C} \to s^{-1}\bar{C} \otimes s^{-1}\bar{C}$ that derives from the comultipliaction $\Delta$ of $C$. 
Since $C$ is cocommutative, then we always express the $d_1(s^{-1}c)$ as the sum of commutators in $Ts^{-1}\bar{C}$. 
Let $\bar{\Delta}c=\sum a_i\otimes b_i$, then $\bar{\Delta}c=\sum{(-1)}^{\text{deg }a_i \text{ deg }b_i} b_i\otimes a_i$. 
So $$d_1 (s^{-1}c) = \frac{1}{2} \sum_{i} {(-1)}^{\text{deg }a_i}[s^{-1}a_i,s^{-1}b_i]$$
through simple calculations, then we can know that $d_1: s^{-1}\bar{C} \to \mathbb{L}_{s^{-1}\bar{C}} \subseteq Ts^{-1}\bar{C}$. 
Hence, we have proven that $d=d_0+d_1$ is the Lie derivation of free Lie algebra $\mathbb{L}_{s^{-1}C}$.
\vspace{0.3cm}
\begin{definition}
    The $dgl$ $(\mathbb{L}_{s^{-1}\bar{C}},d)$ is referred to as the Quillen construction on the co-augmented $cdgc$ $(C,d)$ and it is denoted by $\mathcal{L}(C,d)$. 
\end{definition} 
\vspace{0.3cm}
\begin{theorem}{\cite{felix2012rational}}
    Let $(L=\{L_i\}_{i\geq 1},d)$ be a connected $dgl$ and $(C=\mathbb{Q}\oplus C_{\geq 2},d)$ is a $cdgc$. 
    Then, there exist natural quasi-isomorphisms 
    $$\varphi:(C,d)\to C_*\mathcal{L}(C,d) \text{ and } \psi:\mathcal{L}C_*(L,d) \to (L,d)$$
    of $cdgc$'s (respectively, of $dgl$'s). 
\end{theorem} 
\vspace{0.3cm}
The two functors, $C_*$ and $\mathcal{L}$, we introduced above are adjoint to each other:
$$ \mathcal{L} \dashv C_* .$$
What's more, the adjunction $(\mathcal{L} \dashv C_*)$ is a Quillen adjunction between the projective model structure on $\textbf{DGL}$ and the model structure on $\textbf{CDGC}$. 
    
For the category $\textbf{DGL}$, there is a model category structure $(\textbf{DGL})_{proj}$ on the category $\textbf{DGL}$ over $\mathbb{Q}$ so that 
\begin{itemize}
    \item the fibrations the surjective maps
    \item weak equivalences the quasi-isomorphisms on the underlying chain complexes.
\end{itemize}
Meanwhile, for the category \textbf{CDGC}, there is a model category structure $(\textbf{CDGC})_{Quillen}$ on the category \textbf{CDGC} over $\mathbb{Q}$ so that 
\begin{itemize}
    \item the cofibrations are the (degreewise) injections;
    \item the weak equivalences are those morphisms that become quasi-isomorphisms under the functor $\mathcal{L}$, that is, quasi-isomorphisms if $dgc$ is 1-connected. 
\end{itemize}
Furthermore, Vladimir Hinich proved that the Quillen adjunction $(\mathcal{L} \dashv C_*)$ is a Quillen equivalence\cite{hinich2001dg}.
More generally, Quillen proved the following theorem
\vspace{0.3cm}
\begin{theorem}{\cite{quillen1969rational}}{\label{11}}
    There exist equivalences of categories
    $$\textbf{Ho}(\textbf{Top} )\xrightarrow[]{\lambda} \textbf{Ho}(\textbf{DGL}) \xrightarrow[]{C_*}  \textbf{Ho}(\textbf{CDGC}). $$
\end{theorem}
\vspace{0.3cm}

We define the functor $C^*(-)=\textbf{Hom}(C_*(-),\mathbb{Q})$. 
    Moreover, we have an important fact that $C^*(L)$ is a commutative $cdga$ because $C_*(L)$ is cocommutative. 
    Moreover, if $(L,d_L)$ is connected, then $C_*(L,d_L)=\Lambda sL= \mathbb{Q}\oplus \{C_i\}_{i\geq 2}$. 
    The assertion that $C^*(L,d_L)$ is a Sullivan algebra follows from dualizing the Cartan-Eilenberg-Chevalley construction and leveraging properties of differential graded Lie algebras and Sullivan models. 

    Next, we will introduce the definition of the Quillen model for rational spaces, which is actually a Lie algebra model $(L,d_L)$ for rational spaces $X$ with the property $H_*(L,d_L) \cong (\pi_*(\Omega X) , [\ ,\ ])$ where $[\ ,\ ]$ is determined by the Whitehead product $[\ ,\ ]_{W}$. 
    \vspace{0.3cm}
    \begin{definition}
        A free model of $(L,d) \in \text{ob }\textbf{DGL} $ is a quasi-isomorphism of differential graded Lie algebras 
        $$n:(\mathbb{L}_V,d) \xrightarrow[]{\simeq} (L,d)$$
        with $V=\{V_i\}_{i\geq 1}$. 

        If $(\mathbb{L}_V,d)$ is minimal, we call $m:(\mathbb{L}_V,d) \xrightarrow[]{\simeq} (L,d)$ a minimal free Lie model of $(L,d)$.
    \end{definition}
    \vspace{0.3cm}
    \begin{definition}
        Let $X \in \text{ob }\textbf{Top} $. A Lie model for $X$ is a quasi-isomorphism of differential graded algebras
        $$n_X:C^*(L,d_L) \xrightarrow[]{\simeq} A_{PL}(X).$$
        where $(L,d_L)$ is a connected $dgl$ of finite type.
        Sometimes, we also say that $L$ is the Lie model of $X$. 
        If $L=\mathbb{L}_V$, a free graded Lie algebra, we say $(L,d_L)$ is a free Lie model for $X$. 

        Let $n_Y:C^*(E,d_E) \xrightarrow[]{\simeq} A_{PL}(Y)$ be a Lie model for the space $Y$, and $f:X\to Y$ be a continuous map. 
        Then, a Lie representative for $f$ is a morphism of differential graded Lie algebras
        such that $n_XC^*(\varphi)\sim A_{PL}(f)n_Y$. 
    \end{definition} 
    \vspace{0.3cm}
    In fact, the functor $\lambda : \textbf{Top}  \to \textbf{DGL}$, which is constructed by Quillen, assigns a space $X$ a Lie algebra $\lambda X$ which is a free Lie algebra. 
    Thus, we call $\lambda X$ the Quillen model of $X$, and if a free Lie model $(\mathbb{L}_V,d)$ of $X$ is minimal, then we call $(\mathbb{L}_V,d)$ a minimal free Lie model or minimal Quillen model of $X$. 
    \vspace{0.3cm}
    \begin{example}
        The free Lie model of a sphere $\mathbb{S}^{n+1}$ with $n=2k \text{ or } 2k+1$

        $$\mathbb{L}(v)=
        \begin{cases}
            \mathbb{Q}v, & \text{deg }v=2k\\
            \mathbb{Q}v \oplus \mathbb{Q}[v,v], & \text{deg }v=2k+1. 
        \end{cases}$$
        and $d_L=0$. 
    \end{example}
    \vspace{0.3cm}
    \begin{proposition}{\cite{felix2012rational}}
        Any space $X\in \text{ob }\textbf{Top} $ has a minimal Quillen model $(\mathbb{L}_V,d)$, unique up to isomorphism. 
        Suppose that $m_X:C^*(\mathbb{L}_V) \to A_{PL}(X)$ is the minimal Quillen model of $X$ and $m_Y:C^*(\mathbb{L}_W) \to A_{PL}(Y)$ is the minimal Quillen model of $Y$. 
        For any continuous map $f: X \to Y$, there is a Lie representative $n_{f}: (\mathbb{L}_V,d) \to (\mathbb{L}_W,d)$.
        
    \end{proposition}
    \vspace{0.3cm}
    In rational homotopy theory, the following theorem establishes a correspondence between differential graded Lie algebras and the rational homotopy types of simply connected spaces: 
    \vspace{0.3cm}
    \begin{theorem}{(Quillen's equivalence)}{\cite{quillen1969rational}}
        Every connected differential graded Lie algebra $(L,d_L)$ of finite type serves as a Lie model for a simply connected CW complex $X$ of finite rational type. 
        Furthermore, this association is unique: two such CW complexes are rationally homotopy equivalent if and only if their corresponding differential degraded Lie algebras are quasi-isomorphic. 
    \end{theorem}
    \vspace{0.3cm}

\section{Persistence Minimal Free Lie Models}
\vspace{0.3cm}
\begin{definition}
    Let $\mathbb{X}:(\mathbb{R},\leq) \to \textbf{Top} $ be a rational $\mathbb{R}$-space. 
    The persistence Quillen model of $\mathbb{X}$ is the functor $\lambda \mathbb{X} : (\mathbb{R},\leq) \to \textbf{DGL}$ with $(\lambda \mathbb{X})_t := \lambda \mathbb{X}_t$. 
\end{definition}
\vspace{0.3cm}
Indeed, through Theorem\ref{11}, we can know that $\lambda$ induces a functor $\textbf{Ho}(\textbf{Top} )^{\mathbb{R}} \to \textbf{Ho}(\textbf{DGL})^{\mathbb{R}}$, since the morphism $\varphi$ in $\textbf{Ho}(\textbf{Top} )^{\mathbb{R}}$ is a set of $\{\varphi_a\}_{a\in \mathbb{R}}$, in which all $\varphi_a$ are morphisms in $\textbf{Ho}(\textbf{Top} )$ and $\lambda$ induces the functor from $\textbf{Ho}(\textbf{Top} )$ to $\textbf{Ho}(\textbf{DGL})$\cite{quillen1969rational}. 
\vspace{0.3cm}
\begin{definition}
    Let $\mathbb{X}:(\mathbb{R},\leq) \to \textbf{Top} $ be a rational $\mathbb{R}$-space. 
    The persistence minimal Quillen model of $\mathbb{X}$ is the functor $M_{Qui}(\mathbb{X}) : (\mathbb{R},\leq) \to \textbf{DGL}$ with $M_{Qui}(\mathbb{X})_t$ is the minimal Quillen model of $\mathbb{X}_t$, and for any $s\leq t$, $M_{Qui}(X)_{s\leq t}$ is a Lie representative of $\mathbb{X}_{s\leq t}$. 
\end{definition}
\vspace{0.3cm}
Note that the definition of persistence minimal Quillen model is not well defined because we cannot promise the equation $M_{Qui}(X)_{r\leq t}=M_{Qui}(X)_{s\leq t} \circ M_{Qui}(X)_{r\leq s}$. 
However, if we focus on the homotopy category of \textbf{DGL}, \textbf{Ho}(\textbf{DGL}), then the definition of the persistence minimal Quillen model is meaningful. 
\subsection{The Existence of Persistence Minimal Free Lie Models}
\vspace{0.3cm}
\begin{lemma}
    Let $n_X:C^*(\mathbb{L}_V) \to A_{PL}(X)$ and $n_Y:C^*(\mathbb{L}_W) \to A_{PL}(Y)$ be free Lie models of $X$ and $Y$ respectively. 
    For any continuous map $f: X \to Y$, the Lie representative $n_{f}: (\mathbb{L}_V,d) \to (\mathbb{L}_W,d)$ is unique up to weak equivalence. 
\end{lemma}
\begin{proof}
    Given the following diagram
    $$
    \xymatrix{
    A_{PL}(Y) \ar[r]^{A_{PL}(f)} & A_{PL}(X) \\
    C^*(\mathbb{L}_W) \ar[u]_{\simeq}^{n_Y} \ar[r]_{C^*(n_f)} & C^*(\mathbb{L}_V) \ar[u]_{n_X}^{\simeq} 
    }
    $$
    is commutative up to homotopy. 
    If there is another Lie representative of $f$, $m_f$, then $C^*(n_f) \sim C^*(m_f)$. 
    Because $C^*(\mathbb{L}_V)$ and $C^*(\mathbb{L}_W)$ are Sullivan models, $C^*(n_f)$ and $C^*(m_f)$ are two Sullivan representatives of $f$, $C^*(n_f) \sim C^*(m_f)$. 
        
    Note that $C_*: \textbf{Ho}(\textbf{DGL}) \to \textbf{Ho}(\textbf{CDGC})$ is a equivalence of categories, $C^*$ induces a equivalence of categories $\textbf{Ho}(\textbf{DGL}) \to \textbf{Ho}(\textbf{CDGA})$ and we still use $C^*$ to represent it. 
    What's more, we know that if two morphisms in \textbf{CDGA} are homotopic, then these two morphisms are equivalent in \textbf{Ho}(\textbf{CDGA}), which is the homotopy category of \textbf{CDGA}, where weak equivalences are quasi-isomorphisms. 

    Therefore $n_f=m_f$ in \textbf{Ho}(\textbf{DGL}). 
\end{proof}
\vspace{0.3cm}
So, for any morphisms $X \xrightarrow[]{f} Y \xrightarrow[]{g} Z$ in $\textbf{Top} $, we have proven that $n_g \circ n_f = n_{gf}$ in \textbf{Ho}(\textbf{DGL}), 
where $n_f:\mathbb{L}_U \to \mathbb{L}_V,\  n_g: \mathbb{L}_V \to \mathbb{L}_W,\  n_{gf}:\mathbb{L}_U \to \mathbb{L}_W$ are Lie representatives of $f, g, gf$ respectively, and $\mathbb{L}_U,\mathbb{L}_V,\mathbb{L}_W$ are minimal Quillen models of $X,Y,Z$ respectively. 
\vspace{0.3cm}
\begin{theorem}
    For any rational $\mathbb{R}$-space $\mathbb{X}: (\mathbb{R},\leq) \to \textbf{Top} $, there exists a persistence minimal Quillen model $M_{Qui}(\mathbb{X}):(\mathbb{R},\leq) \to \textbf{Ho}(\textbf{DGL})$ such that $M_{Qui}(\mathbb{X})_t$ is a minimal Quillen model of $\mathbb{X}_r$ and $M_{Qui}(\mathbb{X})_{s\leq t}$ is a Lie representative of $\mathbb{X}_{s \leq t}$ up to weak equivalences. 
\end{theorem}
\vspace{0.3cm}

\subsection{The Stability of Persistence Minimal Free Lie Models}
For the persistence minimal Quillen model we construct, the post-composition of $H_*$ and $\pi_*$ computing the lower bound of the persistence minimal Quillen model $M_{Qui}(\mathbb{X})$ respectively, that is $H_*(M_{Qui}(\mathbb{X}))$ and $\pi_*(M_{Qui}(\mathbb{X}))$ are persistence modules that are functors from $(\mathbb{R,\leq}) \to \textbf{grVec} $. 
    Therefore, we can get the bound of persistence minimal Quillen models. 
    For any rational $\mathbb{R}$-space $\mathbb{X}$, we have the persistence minimal Quillen model $M_{Qui}(\mathbb{X})$. 
    Here, we assume that $Q$ is a map from free Lie algebras to vector spaces, $Q(\mathbb{L}_V)=V$. 
    Obviously, any morphism of free Lie algebras $\varphi:\mathbb{L}_V \to \mathbb{L}_W$ can induces a morphism of vector spaces $Q(\varphi): V \to W$ such that the diagram
    $$
    \xymatrix{
    \mathbb{L}_V \ar[d]_{Q} \ar[r]^{\varphi} & \mathbb{L}_W \ar[d]^{Q}\\
    V \ar[r]_{Q(\varphi)}& W
    }
    $$
    is commutative. 
    
    Given $f: \mathbb{X} \to \mathbb{Y}$, then we have commutative diagram
    $$
    \xymatrix{
    M_{Qui}(\mathbb{X}) \ar[r]^{n_f} \ar[d]_{Q} & M_{Qui}(\mathbb{Y}) \ar[d]^{Q}\\
    \mathbb{V} \ar[r]_{Q(n_f)} & \mathbb{W}
    }$$
    where $\mathbb{V}_r:=Q(M_{Qui}(\mathbb{X}_r))$ and $\mathbb{V}_{s\leq t} := Q(M_{Qui}(\mathbb{X})_{s \leq t})$. 
    \vspace{0.3cm}
    \begin{theorem}\label{12}
        For any rational $\mathbb{R}$-spaces $\mathbb{X}$ and $\mathbb{Y}$, we have 
        \begin{itemize}
            \item $d_{I}^{\textbf{Ho}(\textbf{DGL})}(M_{Qui}(\mathbb{X}),M_{Qui}(\mathbb{Y})) \leq d_{HI}(\mathbb{X},\mathbb{Y}) \leq d_{I}(\mathbb{X}, \mathbb{Y})$
            \item $\begin{array}{ll}
            d_I^{\textbf{grVec} }(\pi_*(\mathbb{X}), \pi_*(\mathbb{Y}) )&=d_I^{\textbf{grVec} }(H_* \circ M_{Qui}(\mathbb{X}),H_* \circ M_{Qui}(\mathbb{Y})) \\
            &\leq d_{I}^{\textbf{Ho}(\textbf{DGL})}(M_{Qui}(\mathbb{X}),M_{Qui}(\mathbb{Y}))
            \end{array}$
            \item $d_I^{\textbf{grVec} }(H_*(\mathbb{X}), H_*(\mathbb{Y}) )=d_I^{\textbf{grVec} }(\mathbb{V},\mathbb{W}) \leq d_{I}^{\textbf{Ho}(\textbf{DGL})}(M_{Qui}(\mathbb{X}),M_{Qui}(\mathbb{Y}))$
        \end{itemize}
    \end{theorem}
    \vspace{0.3cm}
    To prove the theorem, we need some extra results.
    \vspace{0.3cm}
    \begin{lemma}{\cite{felix2012rational}}{\label{13}}
        Let $(L,d)$ be a Lie model for $X \in \text{ob }\textbf{Top} $.
        There exists a natural isomorphism $H_*(L) \xrightarrow[]{\cong} \pi_*(\Omega X)$ of graded Lie algebras,
        which converts the Lie bracket in $H_*(L)$ to the Whitehead product in $\pi_*(X)$ up to sign. 
    \end{lemma}
    \vspace{0.3cm}
    For any free Lie algebra $(\mathbb{L}_V,d)$, let $d_V: V \to V$ be the linear part of the differential $d$, and $\bar{d}:sV \to sV$ be the suspension of $d_V$. 
    And for any continuous map $f:X\to Y$, respective free Lie models $(\mathbb{L}_V,d)$ and $(\mathbb{L}_W,d)$ of $X$ and $Y$, and a Lie representative $n_f$ of $f$, we have know that $sH(V,d_V)\oplus \mathbb{Q} \cong H_*(X)$\cite{felix2012rational} and consider the linear part of the Lie representative $n_f$, $Q(n_f): (sV\oplus\mathbb{Q},d_V) \to (sW\oplus\mathbb{Q},d_W)$. 
    
    We naturally pose the question: Is the morphism $H(Q(n_f))$ induced by $Q(n_f)$ ‘equal’ to the morphism $H_*(f)$ ? 
    The following lemma provides an answer to our question. 
    \vspace{0.3cm}
    \begin{proposition}{\label{14}}
        Suppose $(\mathbb{L}_V,d)$ is a free Lie model for $X$, then $sH(V,d_V)\oplus \mathbb{Q} \cong H_*(X)$ is a natural isomorphism of graded vector spaces.
        
        To be more detailed, we have the following commutative diagram. 
        $$
        \xymatrix{
        H_*(X) \ar[r]^{H_*(f)} \ar[d]_{\cong}^{} & H_*(Y) \ar[d]_{}^{\cong} \\
        sH(V,d_V)\oplus \mathbb{Q} \ar[r]_{H(Q(n_f))} & sH(W,d_W)\oplus \mathbb{Q}
        }
        $$
        Specially, if $(\mathbb{L}_V,d)$ is minimal, then $H_*(X)\cong sV\oplus \mathbb{Q}$. 
    \end{proposition}
    \begin{proof}
        First, the morphism $C_*(\mathbb{L}_V,d) \xrightarrow[]{\simeq} A_{PL}(X)$ induces a cohomology isomorphism, that dualizes to an isomorphism $H_*(X) \xrightarrow[]{\cong} H_*(C_*(\mathbb{L}_V),d)$. 
        Given that $n_f$ is a Lie representative of $f:X \to Y$, then we have the following commutative diagram up to homotopy.
        $$
        \xymatrix{
        C^*(\mathbb{L}_V)\ar[d]_{n_X}&C^*(\mathbb{L}_W,d)\ar[d]^{n_Y} \ar[l]_{C^*(n_f)}\\
        A_{PL}(X)&A_{PL}(Y)\ar[l]^{A_{PL}(f)}
        }
        $$
        Thus the diagram
        $$
        \xymatrix{
        H(C^*(\mathbb{L}_V,d))\ar[d]_{\cong}&H(C^*(\mathbb{L}_W,d))\ar[d]^{\cong} \ar[l]_{H\circ C^*(f)}\\
        H(A_{PL}(X))\ar[d]_{\cong}&H(A_{PL}(Y))\ar[l]^{H\circ A_{PL}(f)}\ar[d]^{\cong}\\
        H^*(X) & H^*(Y)\ar[l]^{H^*(f)}
        }
        $$
        is commutative. 
        Then, we get the following commutative diagram. 
        $$
        \xymatrix{
        H(C^*(\mathbb{L}_V,d))\ar[r]^{H\circ C_*(n_f)}&H(C^*(\mathbb{L}_W,d))\\
        H_*(X)\ar[u] & H_*(Y)\ar[u]\ar[l]^{H_*(f)}
        }
        $$

        Note that in \cite{felix2012rational}, one provides a quasi-isomorphism $C_*(\mathbb{L}_V,d) \to (sV\oplus \mathbb{Q},\bar{d})$ for any free Lie algebra $(\mathbb{L}_V,d)$. 
        The quasi-isomorphism $C_*(\mathbb{L}_V,d) \to (sV\oplus \mathbb{Q},\bar{d})$ is 
        $$C_*(\mathbb{L}_V,d) = \Lambda s\mathbb{L}_V \to s\mathbb{L}_V\oplus\mathbb{Q} \to sV\oplus\mathbb{Q}, $$
        where the first morphism annihilates $\Lambda^{\geq 2}s\mathbb{L}_V$ and the second morphism annihilates $s\mathbb{L}_v^{(\geq 2)}$. 
        We have obviously the following commutative diagram
        $$
        \xymatrix{
        C_*(\mathbb{L}_V) \ar[d]^{=}\ar[r]^{C_*(n_f)} & C_*(\mathbb{L}_W)\ar[d]^{=}\\
        \Lambda s\mathbb{L}_V\ar[d] & \Lambda s\mathbb{L}_W\ar[d] \\
        s\mathbb{L}_V\oplus \mathbb{Q} \ar[d]\ar[r]^{n_f} & s\mathbb{L}_W \oplus \mathbb{Q}\ar[d]\\
        sV\oplus \mathbb{Q} \ar[r]^{Q(n_f)} & sW\oplus\mathbb{Q}
        }   
        $$
        So we eventually get the following commutative diagram, which shows that $H_*(X) \xrightarrow[]{\cong} sH(V,d_V)$ is natural. 

        It is also easy to prove that $sH(V,d_V)\oplus \mathbb{Q} \xrightarrow[]{\cong} H_*(X)$ is natural.
    \end{proof}
    \vspace{0.3cm}

    With the two lemmas established above, we can now readily proceed to prove my theorem. 
    \vspace{0.3cm}
    \begin{proof}{ of Theorem\ref{12}. }
        This inequality $d_{HI}(\mathbb{X},\mathbb{Y})\leq d_I(\mathbb{X},\mathbb{Y})$ is obvious and also an existing result. 
        Suppose $d_{HI}(X,Y)=\delta$, then there is persistence spaces $\mathbb{X}' \text{ and } \mathbb{Y}': (\mathbb{R},\leq) \to \textbf{Top} $ such that $\mathbb{X} \simeq \mathbb{X}'$, $\mathbb{Y} \simeq \mathbb{Y}'$, and $d_I(\mathbb{X}',\mathbb{Y}')=\delta$. 
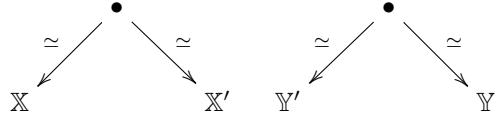
        $$
        \xymatrix{
        &\bullet\ar[ld]_{\simeq}\ar[rd]^{\simeq}&\\
        \mathbb{X}&&\mathbb{X}'
        }
        \text{\quad   }
        \xymatrix{
        &\bullet\ar[ld]_{\simeq}\ar[rd]^{\simeq}&\\
        \mathbb{Y}'&&\mathbb{Y}
        }
        $$
        Consider their persistence minimal Quillen models in $\textbf{Ho}(DGL)$,
        $$
        
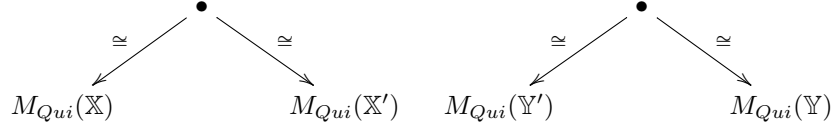
\xymatrix{
        &\bullet\ar[ld]_{\cong}\ar[rd]^{\cong}&\\
        M_{Qui}(\mathbb{X})&&M_{Qui}(\mathbb{X}')
        }
        \text{\quad   }
        \xymatrix{
        &\bullet\ar[ld]_{\cong}\ar[rd]^{\cong}&\\
        M_{Qui}(\mathbb{Y}')&&M_{Qui}(\mathbb{Y})
        }
        $$
        where $M_{Qui}(\mathbb{X})$ is a object in category $\textbf{Ho}(DGL)^{\mathbb{R}}$, $M_{Qui}(\mathbb{X}')$, so are $M_{Qui}(\mathbb{Y})$, and $M_{Qui}(\mathbb{Y}')$.
        
        Suppose that $\mathbb{X}'$ and $\mathbb{Y}'$ are $(\delta+\epsilon)$-interleaved for any $\epsilon >0$, a $(\delta+\epsilon)$-interleaving between $\mathbb{X}'$ and $\mathbb{Y}'$ induces a $(\delta+\epsilon)$-interleaving between $M_{Qui}(\mathbb{X}')$ and $M_{Qui}(\mathbb{Y}')$. 
        Then $M_{Qui}(\mathbb{X})$ and $M_{Qui}(\mathbb{Y})$ are $(\delta+\epsilon)$-interleaved. 
        Thus we have proven that $d_I^{\textbf{Ho}(\textbf{DGL})}(M_{Qui}(\mathbb{X}),M_{Qui}(\mathbb{Y})) \leq d_{HI}(\mathbb{X},\mathbb{Y})$. 

        For the other two inequalities, $d_I^{\textbf{Vec} }(H_* \circ M_{Qui}(\mathbb{X}),H_* \circ M_{Qui}(\mathbb{Y})) \leq d_{I}^{\textbf{Ho}(\textbf{DGL})}(M_{Qui}(\mathbb{X}),M_{Qui}(\mathbb{Y}))$ and $d_I^{\textbf{Vec} }(\mathbb{V},\mathbb{W}) \leq d_{I}^{\textbf{Ho}(\textbf{DGL})}(M_{Qui}(\mathbb{X}),M_{Qui}(\mathbb{Y}))$ are obvious. 
        Lemma\ref{13} show that $d_I^{\textbf{Vec} }(\pi_*(\mathbb{X}), \pi_*(\mathbb{Y}) )=d_I^{\textbf{Vec} }(H_* \circ M_{Qui}(\mathbb{X}))$ and Proposition\ref{14} show that $d_I^{\textbf{Vec} }(H_*(\mathbb{X}), H_*(\mathbb{Y}) )=d_I^{\textbf{Vec} }(\mathbb{V},\mathbb{W})$. 
    \end{proof}
    \vspace{0.3cm}
    From the proof process, we can see that apart from proving $d_I^{\textbf{Vec} }(H_*(\mathbb{X}), H_*(\mathbb{Y}) )=d_I^{\textbf{Vec} }(\mathbb{V},\mathbb{W})$, we did not use the properties of the minimal Quillen model. Therefore, for any persistence free Lie model $\mathbb{L_V}$ and $\mathbb{L_W}$ of rational $\mathbb{R}$-spaces $\mathbb{X}$ and $\mathbb{Y}$ respectively, we have the following results: 
    \begin{itemize}
            \item $d_{I}^{\textbf{Ho}(\textbf{DGL})}(\mathbb{L_V},\mathbb{L_W}) \leq d_{HI}(\mathbb{X},\mathbb{Y})$, 
            \item $d_I^{\textbf{Vec} }(\pi_*(\mathbb{X}), \pi_*(\mathbb{Y}) )=d_I^{\textbf{Vec} }(H_* \circ \mathbb{L_V},H_* \circ \mathbb{L_W}) \leq d_{I}^{\textbf{Ho}(\textbf{DGL})}(\mathbb{L_V},\mathbb{L_W})$, 
            \item $d_I^{\textbf{Vec} }(H_*(\mathbb{X}), H_*(\mathbb{Y}) )\leq d_I^{\textbf{Vec} }(\mathbb{V},\mathbb{W}) \leq d_{I}^{\textbf{Ho}(\textbf{DGL})}(\mathbb{L_V},\mathbb{L_W}) $. 
    \end{itemize}
    What's more, we can prove easily that $d_I^{\textbf{Ho}(\textbf{Top} )}(\mathbb{X},\mathbb{Y})=d_{I}^{\textbf{Ho}(\textbf{DGL})}(\mathbb{L_V},\mathbb{L_W})$. 

    In persistent homology, the persistence free Lie models have some special advantages. 
    \vspace{0.3cm}
    \begin{example}
        Let $\mathbb{X}:(\mathbb{N},\leq)\to \textbf{Top} $ be the filtration of skeletons of CW complex $X$ satisfying $\mathbb{X}_r=X^r$ for $r \geq 2$ and $\mathbb{X}_0=\mathbb{X}_1=\emptyset$, where $X$ is a simply connected $CW$ complex so that $H_*(X;\mathbb{Q})$ is of finite type, and $X^r$ is the $r$-dim skeleton of $X$. 
        We know that $X^{r+1}=X^{r}\cup_{f_r}(\coprod_{\alpha} D_{\alpha}^{r+1})$, in which $f_r:=\coprod_{\alpha}f_{r,\alpha}:\coprod_{\alpha}\mathbb{S}_{\alpha}^{r} \to X^r$. 
        Next, we will construct a persistence free Lie model $Lie(\mathbb{X})$ for $\mathbb{X}$. 
        
        First, define $Lie(\mathbb{X})_0=Lie(\mathbb{X})_1=0$ and $Lie(\mathbb{X})_2=\lambda X^2$. 
        Suppose that we have got $Lie(\mathbb{X})_r$ which is a free Lie model of $X^r$, that is $n_r:C^*(Lie(\mathbb{X})_r) \xrightarrow[]{\simeq} A_{PL}(X^r)$ is a quasi-isomorphism. 
        
        Without loss of generality, we assume that $Lie(\mathbb{X})_r=\mathbb{L}_V$. 
        Because we have the isomorphism 
        $$\tau:sH(\mathbb{L}_V) \xrightarrow[]{\cong} \pi_*(X^r),$$
        then the classes $[f_{r,\alpha}] \in \pi_*(X^r)$ determine the classes $s[z_\alpha] = {\tau}^{-1} [f_{r,\alpha}] \in sH(\mathbb{L}_V)$, where $z_{\alpha} \in \mathbb{L}_V$ are cycles. 

        We define that $Lie(\mathbb{X})_{r+1}:=\mathbb{L}_{V\oplus W}$ and $d w_{\alpha}=z_{\alpha}$, in which $W$ is a graded vector space with basis $\{w_{\alpha}\}$ with $\text{deg }w_{\alpha} = r$. 
        We assert that $\mathbb{L}_{V\oplus W}$ is a free Lie model for $X^{r+1}$\cite{felix2012rational}. 
        Therefore, we define a free Lie model $Lie(\mathbb{X})$ for $\mathbb{X}$, denoted as $\mathbb{L_V}$ with $(\mathbb{L_V})_r=\mathbb{L}_{\mathbb{V}_r}=Lie(\mathbb{X})_r$, where $\mathbb{V}:(\mathbb{N},\leq) \to \textbf{Vec} $ is a persistence module and any morphism $\mathbb{V}_{s\leq t}$ is an embedding.
    \end{example}
    \vspace{0.3cm}
    
    In addition to constructing persistence Lie models, we can also consider the persistence versions of Lie-infinity models\cite{hinich2001dg, buijs2013rational} for rational spaces. 
    Lie-infinity algebras inherently align more closely with the homotopy theory of topological spaces than classical Lie algebras. 
    Indeed, while Quillen’s construction provides a Lie-infinity model for a rational space $X$, bridging the gap to establish persistence Lie-infinity models and discuss their stability properties remains an open challenge. 
    In fact, although Quillen's construction provides a Lie-infinity model for a rational space $X$, we still need a little work to overcome the difficulties if we consider persistence Lie infinite models and the stability of persistence Lie-infinity models. 
    And if we can construct minimal Lie-infinity models\cite{kontsevich2003deformation} for rational $\mathbb{R}$-spaces and prove that this construction satisfies functoriality, then I believe this model will have a unique advantage in theory and application of persistence modules.

\section*{Acknowledgments}
I am grateful to all those who have provided assistance in achieving these results.

% References
\bibliographystyle{plain}

\bibliography{reference}

\end{document}